\documentclass[12pt,reqno]{amsart}
\usepackage{amsmath,amsopn,amssymb,amsthm,multicol}
\usepackage{hyperref}
\usepackage[all]{xy}
\usepackage{amscd}
\usepackage{graphicx}
\usepackage{graphics}
\usepackage{xcolor}

\DeclareMathOperator{\Ad}{Ad}

\DeclareMathOperator{\Ric}{Ric}

\newcommand{\fr}{\mathfrak}

\newcommand{\bb}{\mathbb}

\DeclareMathOperator{\SO}{SO}

\DeclareMathOperator{\Sp}{Sp}
\DeclareMathOperator{\SU}{SU}

 \newtheorem{lemma} {Lemma} [section]
\newtheorem{theorem}[lemma]{Theorem} 
\newtheorem{remark}[lemma] {Remark} 
\newtheorem{prop} [lemma]{Proposition}

\usepackage{array}
\makeatletter
\newcommand{\thickhline}{%
    \noalign {\ifnum 0=`}\fi \hrule height 1pt
    \futurelet \reserved@a \@xhline
}
\newcolumntype{"}{@{\hskip\tabcolsep\vrule width 1pt\hskip\tabcolsep}}
\makeatother

\begin{document}
\title{Non naturally reductive Einstein metrics on the orthogonal group via real flag manifolds}
\author{Andreas Arvanitoyeorgos
, Yusuke Sakane and Marina Statha}
\address{University of Patras, Department of Mathematics, GR-26500 Rion, Greece
AND\newline
Hellenic Open University, Aristotelous 18, GR-26335 Patras, Greece}
\email{arvanito@math.upatras.gr}
 \address{Osaka University, Department of Pure and Applied Mathematics, Graduate School of Information Science and Technology, Toyonaka, Osaka 560-0043, Japan}
 \email{sakane@math.sci.osaka-u.ac.jp}
\address{University of Thessaly, Department of Mathematics, GR-35100 Lamia, Greece}
\email{marinastatha@uth.gr} 
\medskip

\begin{abstract}
  
We obtain new invariant Einstein metrics  on the compact Lie groups
 $\SO(n)$  which are not naturally reductive.  This is achieved by  using the real flag manifolds $\SO(k_1+\cdots +k_p)/\SO(k_1)\times\cdots\times\SO(k_p)$ and by imposing certain symmetry assumptions in the set of all 
 left-invariant metrics on $\SO(n)$.


   \end{abstract}

\maketitle
 
 
 
 


\section{Introduction}
\markboth{Andreas Arvanitoyeorgos, Yusuke Sakane and Marina Statha}{New homogeneous Einstein metrics on quaternionic Stiefel manifolds}

A Riemannian manifold $(M, g)$ is called Einstein if it has constant Ricci curvature, i.e. $\Ric_{g}=\lambda\cdot g$ for some $\lambda\in\bb{R}$. 
  A detailed exposition on Einstein manifolds can be found in \cite{Be} and various  results about homogeneous Einstein metrics are contained in \cite{W1}, \cite{W2} and \cite{A}. 
For the case of homogeneous spaces $G/K$ one needs to find and classify all $G$-invariant Einstein metrics, and for  a 
compact Lie group the problem is to find left-invariant Einstein metrics.
The later is more subtle, since the space of all left-invariant metrics on a Lie group up to isometry and scaling is quite difficult to study. As a result, even for low dimensional Lie groups such as  $\SU(3)$ and $\SU(2)\times\SU(2)$, the number of left-invariant Einstein metrics  is still unknown.

It is well known that a compact and semisimple Lie group equipped with a bi-invariant metric is Einstein.
 In \cite{DZ}  J.E. D'Atri and W. Ziller found a large number of naturally reductive metrics on the compact Lie groups and
  they raised the question of existence of left-invariant Einstein metrics which are not naturally reductive.
  A first answer was given  by K. Mori in 
  \cite{M}, who obtained non naturally reductive Einstein metrics on the Lie group $\SU(n)$ for $n\ge 6$.  Since then, several authors obtained non naturally reductive Einstein metrics on other classical and exceptional Lie groups (e.g. \cite{AMS},  \cite{ASS2}, \cite{ASS4}, \cite{ASS5}, \cite{CL}, \cite{CCD}, ).
  
A method to investigate  left-invariant Einstein metrics on a compact Lie group $G$ is the following.  
  We consider a homogeneous space $G/K$ and decompose the tangent space $\fr{g}$ of $G$ via the submersion $G\to G/K$ with fiber $K$, 
  into a direct sum of irreducible $\Ad(K)$-modules.  Then 
  we consider left-invariant metrics on $G$ which are determined by diagonal $\Ad(K)$-invariant scalar products on $\fr{g}$. By using the bracket relations of the submodules, we can see that  Ricci curvature also satisfies  orthogonal relations for each pair of different irreducible summands. Thus,  
  by taking into account the diffeomorphism 
  $G/\{e\}\cong (G\times K)/\mbox{diag}(K)$,  we can use well known formula in \cite{PS} for the Ricci curvature of such left-invariant metrics on $G$.

 The existence of non naturally reductive left-invariant Einstein metrics on the orthogonal group $\SO(n)$ ($n\ge 11$) was originally shown by the first two authors  and K. Mori in \cite{AMS}. 
 There, the Lie group  $\SO(n)$ was considered as a total space over appropriate generalized flag manifolds with two isotropy summands.
More recently, in \cite{ASS2}, the authors  proved existence of  non naturally reductive left-invariant Einstein metrics on $\SO(n)$ for $n\ge 7$, different form the ones in \cite{AMS}.
In that work, we viewed  $\SO(n)$ as a total space over  the homogeneous space $G/K = \SO(k_1+k_2+k_3)/(\SO(k_1)\times\SO(k_2)\times\SO(k_3))$, which is an example of a  generalized Wallach space according to \cite{N}, 
and worked for the cases with small numbers $k_1, k_2, k_3$. 

Similar methods were also used by H. Chen, Z. Chen and D. Deng in \cite{CCD}, and B. Zhang, H. Chen and J. Tan in \cite{ZCT}, where they proved existence of  non naturally reductive left-invariant Einstein metrics on the orthogonal group $\SO(n)$,  for $ k_1 \geq 3, \,  k_2=3, \,  k_3=3$  and  for $ k_1 \geq k,\, k_2=k, \,  k_3=k$  ($k \geq 3$)  respectively.

   In the present paper we  generalize the above approaches and prove existence of new non naturally reductive left-invariant Einstein metrics on the orthogonal group $\SO(n)$, $n=k_1+k_2+\cdots +k_p$  (where $k_i \geq 2$ ($i =1, \dots, p$)),  by using the space 
 $\SO(k_1+\cdots +k_p)/(\SO(k_1)\times\cdots\times\SO(k_p))$.   This space  is known in the literature as a {\it real flag manifold}  (e.g.   \cite{PM}, \cite{R}).
 
 Note that 
 the isotropy representation $\fr{m}=\fr{m}_{12}\oplus\cdots \oplus\fr{m}_{1p}\oplus\fr{m}_{23}\oplus\cdots\oplus\fr{m}_{p-1, p}$ of $G/K$ does not contain equivalent irreducible summands as $\Ad(K)$-modules, except for the case $k_i =k_j =2$ for some $i, j$ ($i \neq j$).  Thus, under these conditions,  
 we see that left-invariant metrics are determined by  $\Ad(K)$-invariant inner products   of the form
 \begin{equation*}
  \langle \  \ ,\ \ \rangle =  x_1 \, (-B) |_{\fr{so}(k_1)}+ x_2 \, (-B) |_{ \fr{so}(k_2)}+ \cdots + x_p \, (-B) |_{ \fr{so}(k_p)} 
 + \sum_{1\le i<j\le p} x_{ij} \,  (-B) |_{ \fr{m}_{ij}}. 
 \end{equation*}


 We take $k_2=\cdots =k_p=k\ge 3$. 
Then the main result  is the following (see Theorems \ref{thm6.2} and \ref{thm6.3}):

\begin{theorem}\label{main}  Let $p\ge 3$. 

\noindent {\rm 1)} 
If $k_1 > k \ge 3$, then  the Lie group  $\SO(n)$ {\em ($n=k_1 +(p-1)k$)} admits at least two  $\Ad( \SO(k_1)\times(\SO(k))^{p-1})$-invariant Einstein metrics, which are not naturally reductive. 
Moreover, if ${k_{1}} \geq 10 k p$, then the Lie group  $\SO(n)$ {\em ($n=k_1 +(p-1)k$) }admits at least four non naturally reductive $\Ad( \SO(k_1)\times(\SO(k))^{p-1})$-invariant Einstein metrics.  

\noindent
{\rm 2)} If $k_1=k$, then the Lie group  $\SO(n)$ {\em ($n=p \,k$)} admits at least one $\Ad( (\SO(k))^{p})$-invariant Einstein metric, which is not naturally reductive.

\end{theorem}

\smallskip

\noindent
The contribution of the above theorem, in addition to those in \cite{ASS2, CCD} and \cite{ZCT}, is that we consider the general space $\SO(n)/(\SO(k_1)\times\cdots\times\SO(k_p))$, whose isotropy representation decomposes into a number of isotropy summands that depends on  $p$, thus viewing $p$ as an additional  parameter in the algebraic system of Einstein equation.  As a consequence, it is possible to obtain a large number of new non-naturally reductive left-invariant Einstein metrics on $\SO(n)$.
  Of course, for $p \geq 4$ these metrics we obtain are different from the ones obtained in \cite{ASS2, CCD} and \cite{ZCT}.

\medskip
The paper is organized as follows:  In Section 2 we recall a general expression for the Ricci tensor for homogeneous spaces whose isotropy representation is decomposed into irreducible non equivalent summands. In Section 3 we consider a special decomposition of the Lie algebra $\fr{so}(n)$ of the special orthogonal group $G=\SO(n)=\SO(k_1+\cdots +k_p)$ by using the real flag manifold $G/(\SO(k_1)\times\cdots\times\SO(k_p))$, and then use this to describe a special class of left-invariant metrics on $G$.  In Section 4 we provide a necessary and sufficient condition so that these left-invariant metrics are naturally reductive.  In Section 5 we compute the Ricci tensor for the special class of left-invariant metrics on     $G=\SO(n)=\SO(k_1+\cdots +k_p)$ and in Section 6 we make an extensive analysis to prove existence of left-invariant Einstein metrics.  To this end, we study a non linear system of equations whose coefficients are parameters depending on  $k_1,  k, p$.

\section{The Ricci tensor for reductive homogeneous spaces}
 We recall an expression for the Ricci tensor for a $G$-invariant Riemannian
metric on a reductive homogeneous space whose isotropy representation
is decomposed into a sum of non equivalent irreducible summands.

Let $G$ be a compact semisimple Lie group, $K$ a connected closed subgroup of $G$  and  
let  $\frak g$ and $\frak k$  be  the corresponding Lie algebras. 
The Killing form $B$ of $\frak g$ is negative definite, so we can define an $\mbox{Ad}(G)$-invariant inner product $-B$ on 
  $\frak g$. 
Let $\frak g$ = $\frak k \oplus
\frak m$ be a reductive decomposition of $\frak g$ with respect to $-B$ so that $\big[\,\frak k,\, \frak m\,\big] \subset \frak m$ and
$\frak m\cong T_o(G/K)$.  
 We decompose  $ {\frak m} $ into irreducible non equivalent $\mbox{Ad}(K)$-modules as follows: \ 
\begin{equation}\label{iso}
{\frak m} = {\frak m}_1 \oplus \cdots \oplus {\frak m}_q.
\end{equation} 
Then for the decomposition (\ref{iso}) any $G$-invariant metric on $G/K$ can be expressed as  
\begin{eqnarray}
 \langle  \,\,\, , \,\,\, \rangle  =  
x_1   (-B)|_{\mbox{\footnotesize$ \frak m$}_1} + \cdots + 
 x_q   (-B)|_{\mbox{\footnotesize$ \frak m$}_q},  \label{eq2}
\end{eqnarray}
for positive real numbers $(x_1, \dots, x_q)\in \mathbb{R}^{q}_{+}$. 
If, for the decomposition (\ref{iso}) of $\frak m$,  the Ricci tensor $r$ of a $G$-invariant Riemannian metric $ \langle  \,\,\, , \,\,\, \rangle$ on $G/K$  satisfies $r({\frak m}_i, {\frak m}_j) = (0)$ for $ i \neq j$, then the Ricci tensor $r$ is of the same form as (\ref{eq2}), that is 
 \[
 r=z_1 (-B)|_{\mbox{\footnotesize$ \frak m$}_1}  + \cdots + z_{q} (-B)|_{\mbox{\footnotesize$ \frak m$}_q} ,
 \]
 for some real numbers $z_1, \ldots, z_q$.

Let $\lbrace e_{\alpha} \rbrace$ be a $(-B)$-orthonormal basis 
adapted to the decomposition of $\frak m$,    i.e. 
$e_{\alpha} \in {\frak m}_i$ for some $i$, and
$\alpha < \beta$ if $i<j$. 
We put ${A^\gamma_{\alpha
\beta}}= -B \big(\big[e_{\alpha},e_{\beta}\big],e_{\gamma}\big)$ so that
$\big[e_{\alpha},e_{\beta}\big]
= \displaystyle{\sum_{\gamma}
A^\gamma_{\alpha \beta} e_{\gamma}}$ and set 
$\displaystyle{k \brack {ij}}=\sum (A^\gamma_{\alpha \beta})^2$, where the sum is
taken over all indices $\alpha, \beta, \gamma$ with $e_\alpha \in
{\frak m}_i,\ e_\beta \in {\frak m}_j,\ e_\gamma \in {\frak m}_k$ (cf.\,\cite{WZ}).  
Then the positive numbers $\displaystyle{k \brack {ij}}$ are independent of the 
$(-B)$-orthonormal bases chosen for ${\frak m}_i, {\frak m}_j, {\frak m}_k$,
and 
$\displaystyle{k \brack {ij}}\ =\ \displaystyle{k \brack {ji}}\ =\ \displaystyle{j \brack {ki}}.  
 \label{eq3}
$

Let $ d_k= \dim{\frak m}_{k}$. Then we have the following:

\begin{lemma}\label{ric2}{\em(\cite{PS})}
The components ${ r}_{1}, \dots, {r}_{q}$ 
of the Ricci tensor ${r}$ of the metric $ \langle  \,\,\, , \,\,\, \rangle $ of the
form {\em (\ref{eq2})} on $G/K$ are given by 
\begin{equation}
{r}_k = \frac{1}{2x_k}+\frac{1}{4d_k}\sum_{j,i}
\frac{x_k}{x_j x_i} {k \brack {ji}}
-\frac{1}{2d_k}\sum_{j,i}\frac{x_j}{x_k x_i} {j \brack {ki}}
 \quad (k= 1, \dots, q),    \label{eq51}
\end{equation}
where the sum is taken over $i, j =1, \dots , q$.
\end{lemma} 
If, for the decomposition (\ref{iso}) of $\frak m$,  the Ricci tensor $r$ of a $G$-invariant Riemannian metric $ \langle  \,\,\, , \,\,\, \rangle$ on $G/K$  satisfies $r({\frak m}_i, {\frak m}_j) =0$ for $ i \neq j$, 
then, by Lemma \ref{ric2}, it follows that $G$-invariant Einstein metrics on $M=G/K$ are exactly the positive real solutions $(x_1, \ldots, x_q)\in\mathbb{R}^{q}_{+}$  of the system of equations $\{r_1=\lambda, \, r_2=\lambda, \, \ldots, \, r_{q}=\lambda\}$, where $\lambda\in \mathbb{R}_{+}$ is the Einstein constant.

\section{A class of left-invariant metrics on $\SO(n)=\SO(k_1+k_2+\dots +k_p)$}

We will describe a decomposition of the tangent space of the Lie group $\SO (n)$ which will be convenient for our study.
 We consider the closed subgroup $K = \SO(k_1)\times\SO(k_2)\times\cdots\times\SO(k_p)$ of $G =\SO(n) =  \SO(k_1 + \cdots+ k_p)$
 with $k_1, k_2, \dots,  k_p \ge 3$,
 where the embedding of $K$ in $G$ is diagonal.
The homogeneous space $G/K$ obtained is known as real flag manifold, which 
for $p=3$ is an example of a generalized Wallach space.  
  Then the tangent space $\fr{so}(k_1 + \cdots+ k_p)$ of the orthogonal group $G = \SO(k_1 + \cdots + k_p)$ can be written as a direct sum of  $\Ad(K)$-invariant modules as 
  \begin{equation}\label{diaspasi}
\fr{so}(k_1 + \cdots + k_p) = \fr{so}(k_1)\oplus\fr{so}(k_2)\oplus\cdots\oplus\fr{so}(k_p)\oplus\fr{m}, 
\end{equation}
where $\fr{m}$ corresponds to the tangent space of $G/K$. 

For $i=1, \dots , p$,  we embed the Lie subalgebras $\fr{so}(k_i)$ in the Lie algebra
$\fr{so}(k_1+\cdots +k_p)$ diagonally.
   Then the tangent space $ \fr{m}$ of $G/K$ 
   is given by  $\fr{k}^{\perp} $ in $ \fr{g} = \fr{so}(k_1+ \cdots+k_p)$ with respect to the $\mbox{Ad}(G)$-invariant inner product $-B$. If we denote by $M(p,q)$ the set of all $p \times q$ real matrices, then we see that 
  $ \fr{m}$ is given by

\begin{equation*}
 \fr{m}=\left\{\left( 
 \begin{array}{ccccc}
  0 & A_{12} &  A_{13}  & \cdots & A_{1p} \\ 
  -A_{12}^t & 0 &  A_{23}  & \cdots & A_{2p}\\
   -A_{13}^t & -A_{23}^t &  0  & \cdots & A_{3p}\\
    \vdots & \vdots &  \vdots  & \ddots & \vdots\\
   -A_{1p}^t & -A_{2p}^t &  \cdots  & -A_{p-1,p}^t & 0
  \end{array}
  \right): A_{ij}\in M(k_i, k_j)
  \right\}.
 \end{equation*}

Now we set

\begin{equation*}
 \fr{m}_{ij}=\left\{\left( 
 \begin{array}{cccc}
  0 & \cdots &    & 0 \\ 
 \vdots & \ddots & A_{ij} & \vdots\\
   & -A_{ij}^t &  & \\
  0 & \cdots &  & 0 
  \end{array}
  \right): A_{ij}\in M(k_i, k_j)
  \right\}.
 \end{equation*}

 
 
 Note that the subspaces  $\fr{m}_{ij}$ ($1\le i<j\le p$)  are 
 irreducible $\Ad(K)$-sub\-modules  
whose dimensions  are $\dim\fr{m}_{ij} = k_{i}k_{j}$.
They are given as orthogonal complements, with respect to the negative of Killing form, of $\fr{so}(k_i)\oplus\fr{so}(k_j)$ in $\fr{so}(k_i+k_j)$ ($1\le i<j\le p$). 

Note that the irreducible submodules $\fr{m}_{ij}$ are mutually non equivalent, so any
 $G$-invariant metric on $G/K$ is determined by an $\Ad(K)$-invariant scalar product
$
\sum_{1\le i<j\le p}x_{ij} \,  (-B) |_{ \fr{m}_{ij}}
$, where $x_{ij}$ are positive real numbers.

For $i=1, \dots , p$ we also set $\fr{m}_i = \fr{so}(k_i)$. 
Therefore, decomposition (\ref{diaspasi}) of the tangent space of the orthogonal group $G= \SO(k_1+\cdots+k_p)$ takes the form
\begin{equation}\label{decom_so(n)}
 \fr{so}(k_1+\cdots +k_p) = \fr{m}_1 \oplus \fr{m}_2 \oplus \cdots \oplus\fr{m}_p \oplus  \bigoplus_{1\le i<j\le p}\fr{m}_{ij}. 
\end{equation}
Then we  see that the following relations hold:
\begin{lemma}\label{brackets}  The submodules in the decomposition {\em (\ref{decom_so(n)})}  satisfy the following bracket relations:
\begin{center}
\begin{tabular}{lll}
$[ \fr{m}_i, \fr{m}_i] = \fr{m}_i,$  &   $[ \fr{m}_i, \fr{m}_{ij}] = \fr{m}_{ij},$  & $[ \fr{m}_j, \fr{m}_{ij}] = \fr{m}_{ij},$ \\
$[ \fr{m}_{ij}, \fr{m}_{jk}] = \fr{m}_{ik},$   &   $[ \fr{m}_{ij}, \fr{m}_{ik}] = \fr{m}_{jk},$  & 
$[\fr{m}_{ik}, \fr{m}_{jk}] =  \fr{m}_{ij},$\\ 
$[ \fr{m}_{ij}, \fr{m}_{ij}] \subset \fr{m}_{i}+\fr{m}_j,$ &   &  
\end{tabular}
\end{center}  
for any $1\le i<j<k\le p$ and the other bracket relations are zero.
  \end{lemma}

Then by taking into account the diffeomorphism
$$
G/\{e\}\cong (G\times \SO(k_1)\times\cdots\times\SO(k_p))/{\rm diag}(\SO(k_1)\times\cdots\times\SO(k_p))
$$ 
we consider left-invariant metrics on $G$ which are determined by the  
  $\Ad(\SO(k_1)\times\cdots\times\SO(k_p))$-invariant scalar products on $\fr{so}(k_1+\cdots+k_p)$ given by
 \begin{equation} \label{metric001} 
 \langle \  \ ,\ \ \rangle =  \sum_{i=1}^p x_i \, (-B) |_{\fr{so}(k_i)}+ 
  \sum_{1\le i<j\le p} x_{ij}\, (-B) |_{\fr{m}_{ij}},
\end{equation}
  where $x_i, x_{ij}>0$.
   Then we see that possible non zero symbols (up to permutation of indices) are 
 \begin{equation}\label{triplets}
 \displaystyle 
{i \brack {ii}}, \ \  {(ij) \brack {i(ij)}}, \ \   {(ij) \brack {j(ij)}},\ \   
    {(ik) \brack {(ij)(jk)}}.
\end{equation}

Denote by $d_i$ and $d_{ij}$ the dimensions of the modules $\fr{m}_i$ and $\fr{m}_{ij}$ respectively.
Then it is $d_i=k_i(k_i-1)/2$,  $d_{ij}=k_ik_j$.

\section{Naturally reductive metrics on the compact Lie group $\SO(n)$} 
We recall the main result of D'Atri and Ziller in \cite{DZ}, where they had investigated naturally reductive metrics among left-invariant metrics on compact Lie groups and gave a complete classification in the case of simple Lie groups.
Let $G$ be a compact, connected semisimple Lie group, $L$ a closed subgroup of $G$ and let
$\fr{g}$ be the Lie algebra of $G$ and $\fr{l}$ the subalgebra corresponding to $L$.
We denote by $Q$ the negative of the Killing form of $\fr{g}$.  Then $Q$ is an
$\Ad (G)$-invariant inner product on $\fr{g}$.
Let $\fr{m}$ be an orthogonal complement of $\fr{l}$ with respect to $Q$.  Then we have
$
\fr{g}=\fr{l}\oplus\fr{m}$, $\Ad(L)\fr{m}\subset\fr{m}$.
Let $\fr{l}=\fr{l}_0\oplus\fr{l}_1\oplus\cdots\oplus\fr{l}_p$ be a decomposition of $\fr{l}$ into ideals, where $\fr{l}_0$ is the center of $\fr{l}$ and $\fr{l}_i$ $(i=1,\dots , p)$ are simple ideals of $\fr{l}$.
Let $A_0|_{\fr{l}_0}$ be an arbitrary metric on $\fr{l}_0$.

\begin{theorem}\label{DZ}{\em(\cite{DZ}, Theorem 1, p. 92) } Under the notations above, a left-invariant metric on $G$ of the form
\begin{equation}\label{natural}
\langle\ ,\ \rangle =x\cdot Q|_{\fr{m}}+A_0|_{\fr{l}_0}+u_1\cdot Q|_{\fr{l}_1}+\cdots
+ u_p\cdot Q|_{\fr{l}_p}, \quad (x, u_1, \dots , u_p >0)
\end{equation}
is naturally reductive with respect to $G\times L$, where $G\times L$ acts on $G$ by
$(g, l)y=gyl^{-1}$.

Moreover, if a left-invariant metric $\langle\ ,\ \rangle$ on a compact simple Lie group
$G$ is naturally reductive, then there is a closed subgroup $L$ of $G$ and the metric
$\langle\ ,\ \rangle$ is given by the form {\em (\ref{natural})}.
\end{theorem}

\medskip
For the Lie group  $\SO(n)$ we consider  $\Ad(\SO(k_1)\times\cdots\times\SO(k_p))$-invariant metrics of the form (\ref{metric001}). 
Recall that $K = \SO(k_1)\times\cdots\times\SO(k_p)$ with Lie algebra $\fr{k}$.

We   consider 
the case when  $k_1 \ge 3$ and $ k_2=\cdots =k_p =k\ge 3$, so that $n=k_1+k(p-1)$, and 
 the left-invariant metric $ \langle \  \ ,\ \ \rangle $ on $\SO(n)$ of the form {(\ref{metric001})}  with  
 $$ \ x_2=\cdots =x_p,\  x_{12} = \cdots =x_{1p},\  x_{23}=x_{ij}\ \ \mbox{for}\ 2\le i<j\le p, $$  
 that is, 
  \begin{equation} \label{metric009} 
 \langle \  \ ,\ \ \rangle = x_1 \, (-B) |_{\fr{so}(k_1)} + x_2 \sum_{i=2}^p  (-B) |_{\fr{so}(k_i)}+  x_{12}\sum_{2\le j\le p}  (-B) |_{\fr{m}_{1j}} + x_{23} \sum_{2\le i<j\le p} (-B) |_{\fr{m}_{ij}}. 
\end{equation}

We need the following result  to prove Theorem \ref{thm6.2} in Section 6. 
\begin{prop}\label{prop5.1}
If a left invariant metric $\langle\ \ ,\ \ \rangle$ of the form {\em (\ref{metric009})} on $\SO(n)$ is naturally reductive  with respect to $\SO(n)\times L$ for some closed subgroup $L$ of $\SO(n)$, 
then  one of  the following holds: 

{\em  1)} the metric $\langle\ \ ,\ \ \rangle$ is bi-invariant. 

{\em  2)} $x_2=x_{23}$, 
 
 \,   {\em  3)}   $x_{12}=x_{23}$. 

Conversely, 
 if   one of the conditions  
 {\em 1)},   {\em 2)},  {\em 3)}, 
 is satisfied, then the metric 
 $\langle\,\, , \,\, \rangle$ of the form {\em (\ref{metric009})}  is  naturally reductive  with respect to $\SO(n)\times L$, for some closed subgroup $L$ of $\SO(n)$.
  \end{prop}

\begin{proof}   Let ${\frak l}$ be the Lie algebra of  $L$. Then we have that either ${\frak l} \subset {\frak k}$  or ${\frak l} \not\subset {\frak k}$. 
First we consider the case of  ${\frak l} \not\subset {\frak k}$. Let ${\frak h}$ be the subalgebra of ${\frak g}$ generated by ${\frak l}$ and ${\frak k}$. 
Since 
$\fr{so}(k_1+\cdots+k_p) = \fr{m}_1\oplus \cdots \oplus\fr{m}_p\oplus  \fr{m}_{12}\oplus\fr{m}_{13}\oplus  \cdots\oplus  \fr{m}_{p-1, p}$ is an irreducible decomposition as $\mbox{Ad}(K)$-modules, we see that the Lie algebra $\frak h$  contains  at least one of  $\fr{m}_{ij}$ $(1\le i<j\le p)$. 

If $\frak h$  contains  one of ${\frak m}_{1j}$ ($2 \leq j$), then  
$\big[{\frak m}_{1j}, {\frak m}_{1j}\big] \subset {\frak m}_1 \oplus \fr{m}_j$ and $\fr{m}_1 \oplus {\frak m}_j  \oplus {\frak m}_{1j}$ is the subalgebra $\fr{so}(k_1+k_j)$. Thus, we see that $x_1 = x_2 =x_{12}$. Note also that  ${\frak m}_{1k} \oplus \fr{m}_{jk}$ ($k \neq 1, j$) is irreducible as $\mbox{Ad}(\SO(k_1+k_j))$-module.  Thus we see  that $x_{12} = x_{23}$ and hence   the metric 
 $\langle\,\, , \,\, \rangle$ is bi-invariant in this case. 

If $\frak h$  contains  one of ${\frak m}_{ij}$ ($2 \leq i < j$), then $\big[{\frak m}_{ij}, {\frak m}_{ij}\big] \subset {\frak m}_i \oplus \fr{m}_j$ and $\fr{m}_i \oplus {\frak m}_j  \oplus {\frak m}_{ij}$ is the subalgebra $\fr{so}(k_i+k_j)$. Thus we see that $x_2  =x_{23}$. 

Now we consider the case ${\frak l} \subset {\frak k}$.  Since the  orthogonal complement
 ${\frak l}^{\bot}$ of ${\frak l}$ with respect to $-B$ contains the  orthogonal complement 
${\frak k}^{\bot}$ of ${\frak k}$, we see that ${\frak l}^{\bot} \supset \displaystyle{\oplus_{1\le i<j\le p}{\frak m}_{ij}}$.   
Since the  invariant metric $\langle \,\, , \,\, \rangle$ is naturally reductive  with respect to $G\times L$,  
 it follows that  $x_{12}=x_{23}$ by  Theorem \ref{DZ}.   
The converse is a direct consequence of Theorem \ref{DZ}.
\end{proof}

\section{The Ricci tensor for a class of left-invariant  metrics on $\SO(n)=\SO(k_1+\dots +k_p)$} 

We will compute the Ricci tensor for the left-invariant metrics on $\SO(n)=\SO(k_1+\cdots+k_p)$, determined by the
$\Ad(K) =\Ad( \SO(k_1)\times\cdots\times\SO(k_p))$-invariant scalar products of the form 
(\ref{metric001}). Note that the Ricci tensor $r$ of the metric (\ref{metric001}) is also  
$\Ad( K )$-invariant.  

For a moment we write the $\Ad( K )$-invariant irreducible  decomposition (\ref{decom_so(n)}) as
\begin{equation*}
\fr{so}(n) = \fr{so}(k_1+\cdots+k_p) = \fr{w}_1 \oplus \fr{w}_2 \oplus \cdots \oplus  \fr{w}_{\frac{p(p+1)}{2}},
\end{equation*}
where $ \fr{w}_i = {\fr m}_i$ for $i = 1, 2, \dots ,p$ and $ {\fr w}_{p+1}, \dots ,\fr{w}_{\frac{p(p+1)}{2}}$ are the modules $\fr{m}_{ij}$ $(1\le i<j\le p)$.
 Then  by Lemma \ref{brackets} it is easy to see the following: 

\begin{lemma}\label{inv_tensor}  
For an $\Ad( K )$-invariant symmetric {\em 2}-tensor $\rho$ on $ \fr{so}(k_1+\cdots+k_p)$,  
we have $\rho(\fr{w}_i, \fr{w}_j) = (0) $ for $i \neq j$. In particular, for the Ricci tensor $r$ of the metric {\em (\ref{metric001})}, we have $r(\fr{w}_i, \fr{w}_j) = (0) $ for $i \neq j$.
\end{lemma}

By Lemma \ref{ric2} and by taking into account (\ref{triplets}) we obtain the following:
\begin{prop}\label{lemma5.1}
The components  of  the Ricci tensor ${r}$ for the left-invariant metric $ \langle \  \ ,\ \ \rangle $ on $\SO(n)$ defined by  \em{(\ref{metric001})} are given as follows:  
\begin{equation*}\label{eq13}
\begin{array}{lll} 
r_i &= & \displaystyle{\frac{1}{2 x_i} +
\frac{1}{4 d_i } \biggl({i \brack {ii}}\frac{1}{x_{i}} +\sum_{j=1, j\ne i}^p{i \brack {(ij)(ij)}}  \frac{x_i}{{x_{ij}}^2}\biggr) 
- \frac{1}{2 d_i } \biggl({i \brack {ii}}\frac{1}{x_{i}} + \sum_{j=1, j\ne i}^p {(ij) \brack {i(ij)}}  \frac{1}{{x_{i}}}\biggr)\ }
\\ & &( i=1, \dots , p ),\\  
r_{ij} &= &  \displaystyle{\frac{1}{ 2 x_{ij}} +\frac{1}{2 d_{ij}} \sum_{k=1, k\ne i, j}^p{(ij) \brack {(ik) (jk)}}    \biggl(\frac{x_{ij}}{x_{ik} x_{jk} } -  \frac{x_{ik}}{x_{ij}x_{jk}} } - \frac{x_{jk}}{x_{ij}x_{ik}}  \biggl) 
\\ &&
\\  
& & 
\displaystyle{-\frac{1}{ 2  d_{ij}} \biggl( {i \brack {(ij) (ij)}} \frac{x_i}{{x_{ij}}^2} 
+ {j \brack {(ij) (ij)}}\frac{x_j}{{x_{ij}}^2} \biggl)}\\
\\ & &( 1\leq i < j \leq  p ),\\  
    \end{array} 
\end{equation*}
where $n = k_1+\cdots +k_p$. 
\end{prop}
%
%
We recall the following result  by the first author, V.V. Dzhepko and Yu.G. Nikonorov: 

\begin{lemma}\label{lemma5.20} {\em \! (\cite{ADN1})} For $a, b, c = 1, 2, \dots p$ and $(a - b)(b - c) (c - a) \neq 0$ the following relations hold: 
\begin{equation*}\label{eq14}
\begin{array}{ll} 
 \displaystyle{{a \brack {a a}} = \frac{k_a (k_a -1)(k_a -2)}{2(n -2)} },   &  \displaystyle{{a \brack {(a b) (a b)}} = \frac{k_a  k_b (k_a -1)}{2(n -2)} }, \vspace{3pt}\\ \displaystyle{{b \brack {(a b) (a b)}} = \frac{k_a  k_b (k_b -1)}{2(n -2)} }, & \displaystyle{{(a c) \brack {(a b ) (b c)}} = \frac{k_a  k_b  k_c}{2(n -2)}}.
\end{array} 
\end{equation*}
\end{lemma}

 By using the above lemma, we can now obtain the components of the Ricci tensor for the metrics 
(\ref{metric001}). 
\begin{prop}\label{prop5.4}
The components  of  the Ricci tensor ${r}$ for the left-invariant metric $ \langle \  \ ,\ \ \rangle $ on $\SO(n)$ defined by  \em{(\ref{metric001})}  are given as follows:  
\begin{equation*}\label{eq17}
\begin{array}{rcl} 
r_i &=&  \displaystyle{\frac{k_i-2}{4 (n -2)  x_i}  +\frac1{4(n-2)}\biggr(
\sum_{j=1, j\ne i}^p k_j\frac{x_i}{x_{ij}^2}\biggr)},\quad i=1, \dots , p
\\
\\
r_{ij} &=&   \displaystyle{\frac{1}{ 2 x_{ij}} +
 \sum_{l=1, l\ne i, j}^p\frac{k_l}{4 (n -2)}\biggl(\frac{x_{ij}}{x_{il} x_{jl}} - \frac{x_{il}}{x_{ij} x_{jl}}  - \frac{x_{jl}}{x_{ij} x_{il}}\biggr) } \\  \\
 &- & 
\displaystyle{\frac{(k_i -1)}{4 (n -2)}  \frac{ x_i}{{x_{ij}}^2}-\frac{(k_j -1)}{4 (n -2)}  
\frac{ x_j}{{x_{ij}}^2} }, \quad 1\le i<j\le p,
\end{array} 
\end{equation*}
where $n=k_1+\cdots +k_p$.
\end{prop}

\section{Left-invariant Einstein metrics on $\SO(n)$} 
In this section we consider 
the case where $k_1 \ge3$,  $k_2=\cdots =k_p =k\ge 3$, so that $n=k_1+k(p-1)$.
We consider  the left-invariant metric $ \langle \  \ ,\ \ \rangle $ on $\SO(n)$ of the form {(\ref{metric001})}  with  
 $$ \ x_2=\cdots =x_p,\  x_{12} = \cdots =x_{1p},\  x_{23}=x_{ij}\ \mbox{for}\ 2\le i<j\le p, $$  
 that is, 
  \begin{equation*} 
 \langle \  \ ,\ \ \rangle = x_1 \, (-B) |_{\fr{so}(k_1)} + x_2 \sum_{i=2}^p  (-B) |_{\fr{so}(k_i)}+  x_{12}\sum_{2\le j\le p}  (-B) |_{\fr{m}_{1j}} + x_{23} \sum_{2\le i<j\le p} (-B) |_{\fr{m}_{ij}}. 
\end{equation*}
For these metrics Proposition \ref{prop5.4} reduces to the following:

\begin{prop}\label{prop5.4_1}
The components  of  the Ricci tensor ${r}$ for the left-invariant metric $ \langle \  \ ,\ \ \rangle $ on $\SO(n)$ defined by  \em{(\ref{metric009})} satisfy  $r_2= r_j$ ($3 \leq j\leq p$),  $r_{12} = r_{1j}$ ($2 <  j \leq p$)  and $r_{23} = r_{ij}$ ($2 \leq i  <  j \leq p$). These components are given as follows:  
\begin{equation*}\label{eq18}
\begin{array}{rcl} 
r_1 &=&  \displaystyle{\frac{k_1-2}{4 (n -2)  x_1}  +\frac{k (p-1)}{4(n-2)}
 \frac{x_1}{{x_{12}}^2}},
 \\  \\
r_2 &=&  \displaystyle{\frac{k-2}{4 (n -2)  x_2}  +\frac1{4(n-2)}\biggr(
 k_1\frac{x_2}{{x_{12}}^2}+  k (p-2)\frac{x_2}{{x_{23}}^2}\biggr)},
 \\   \\
r_{12} &=&   \displaystyle{\frac{1}{ 2 x_{12}} - \frac{k (p-2)}{4 (n -2)}\frac{x_{23}}{{x_{12}}^2} }
-
\displaystyle{\frac{(k_1 -1)}{4 (n -2)}  \frac{ x_1}{{x_{12}}^2}-\frac{(k -1)}{4 (n -2)}  
\frac{ x_2}{{x_{12}}^2} }, 
\\ \\
r_{23} &=&   \displaystyle{\frac{1}{ 2 x_{23}} +\frac{k_1}{4 (n -2)}\biggr(\frac{x_{23}}{{x_{12}}^2}  - \frac{2}{x_{23}}\biggr) }
-  \displaystyle{\frac{k (p-3)}{4 (n -2)}\frac{1}{  x_{23}} }-
\displaystyle{\frac{(k -1)}{2 (n -2)}  \frac{ x_2}{{x_{23}}^2}}.  
\end{array} 
\end{equation*}
\end{prop}

We can now prove one of the main results.

\begin{theorem}\label{thm6.2}
 Let $p\ge 3$ and $k\geq 3$.  If $k_1 > k\ge 3$, then  the Lie group  $\SO(n)$ {\em ($n=k_1 +(p-1)k$) }admits at least two  $\Ad( \SO(k_1)\times(\SO(k))^{p-1})$-invariant Einstein metrics of the form {\em (\ref{metric009})}, which are not naturally reductive. If $k_1 = k \ge 3$, then the Lie group  $\SO(n)$ {\em ($n=p k$) }admits at least one $\Ad( (\SO(k))^{p})$-invariant Einstein metric which is not naturally reductive.
\end{theorem}
    \begin{proof}
     We consider the system of equations 
 \begin{equation}\label{eq26} 
r_1= r_2, \quad r_2 =  r_{12}, \quad  r_{2} = r_{23}. 
 \end{equation}
 Then finding Einstein metrics of the form (\ref{metric009})  reduces  to finding positive solutions of the  system (\ref{eq26}). 
 We consider  our equations by putting $x_{12} =1$.   
  
 Then the system of equations (\ref{eq26}) reduces to  
 the following system of algebraic equations: 
 \begin{equation}\label{eq26g}
\left\{ { \begin{array}{lll}
g_1&  = & k (p-1) {x_1}^2 {x_{2}} {x_{23}}^2-k (p-2)
   {x_1}  {x_{2}}^2-(k-2) {x_1}
    {x_{23}}^2\\&&-{k_{1}} {x_1}
   {x_{2}}^2 {x_{23}}^2+({k_{1}}-2) 
   {x_{2}} {x_{23}}^2=0,\\
g_2 &= & -2 {x_{2}}
   {x_{23}}^2 (k p-k+{k_{1}}-2)+{x_{2}}^2
   {x_{23}}^2 (k+{k_{1}}-1)+k (p-2) 
   {x_{2}}^2\\&&+k (p-2) {x_{2}} {x_{23}}^3+(k-2)
   {x_{23}}^2+({k_{1}}-1) {x_1}
   {x_{2}} {x_{23}}^2-{x_{2}}^2
   {x_{23}}^2,
 \\
g_3& =&  ({x_{2}}-{x_{23}})
   \big( (k p-2){x_{2}} -(k-2) 
   {x_{23}}+{k_{1}} {x_{2}}
   {x_{23}}^2\big)=0. 
\end{array} } \right.
\end{equation}
Note that, if ${x_{2}}={x_{23}}$, we see that $\Ad( \SO(k_1)\times(\SO(k))^{p-1})$-invariant metrics of the form (\ref{metric009}) are naturally reductive metrics by Proposition \ref{prop5.1}. Thus we consider the case when  $ \,x_2 \neq x_{23}$ and $x_{23} \neq 1 =  x_{12}$.  From the equation $g_3=0$ we have 
 \begin{equation}\label{sakane:eqx_2}
 \begin{array}{l}  
 \displaystyle
{x_2}= \frac{(k-2) {x_{23}} }{k p -2+{k_1}{x_{23}}^2}.
\end{array}
\end{equation}
Note that, from (\ref{sakane:eqx_2}), we have $x_2 \neq x_{23}$. 
By substituting (\ref{sakane:eqx_2}) into the system (\ref{eq26g}), we obtain the system of equations
 \begin{equation}\label{eq26b}
\left\{
{ \begin{array}{lll}
f_1& =& {x_{1}} \big(-k^2 ((p-1)) (p+2)+2 k (3
   p-2)-4\big)+k {k_{1}} (p-1) {x_{1}}^2
   {x_{23}}^3\\&& -{k_{1}} {x_{1}} {x_{23}}^2 (2 k
   p+k-6)+({k_{1}}-2) {x_{23}} (k p-2)\\&&+k (p-1)
   {x_{1}}^2 {x_{23}} (k p-2) -{k_{1}}^2 {x_{1}}
   {x_{23}}^4+({k_{1}}-2) {k_{1}}
   {x_{23}}^3=0,\\
  f_2& =& 
   {x_{23}}^2 \big(k^2 (p-1)^2+2 k
   {k_{1}} p+k {k_{1}}-2 k p+k-6
   {k_{1}}+2\big)\\&&
   +k^2 (p-1) (p+2)+({k_{1}}-1)
   {x_{1}} {x_{23}} (k p-2)+{k_{1}} {x_{23}}^4 (k
   p-2 k+{k_{1}})\\ && -2 {k_{1}} {x_{23}}^3 (k
   p-k+{k_{1}}-2)-2 {x_{23}} (k p-2) (k
   p-k+{k_{1}}-2) \\&& +({k_{1}}-1) {k_{1}}
   {x_{1}} {x_{23}}^3-6 k p+4 k+4=0. 
  \end{array}}
  \right.
\end{equation}
From the equation $f_2=0$, we have 
 \begin{equation}\label{eq_x1_rat}
 \begin{array}{l}  
 \displaystyle
{x_1}= (-{x_{23}}^2 \big(k^2
   (p-1)^2+2 k {k_{1}} p+k {k_{1}}-2 k p+k-6
   {k_{1}}+2\big)\\
   -k^2 (p-1) (p+2)-{k_{1}}
   {x_{23}}^4 (k p-2 k+{k_{1}})+2 {k_{1}}
   {x_{23}}^3 (k p-k+{k_{1}}-2)\\+2 {x_{23}} (k p-2)
   (k p-k+{k_{1}}-2)+6 k p-4 k-4)/\\(({k_{1}}-1)
   {x_{23}} (k p-2)+({k_{1}}-1) {k_{1}}
   {x_{23}}^3).
\end{array}
\end{equation}
By substituting (\ref{eq_x1_rat}) into  $f_1 $, we obtain the equation 
 \begin{equation*}\label{eqx23_deg8}  
{ \small \begin{array}{l}
H_1(x_{23})=  {k_{1}}^2 (p k-2 k+{k_{1}}) \big(p^2 k^2-3 p k^2+2
   k^2-{k_{1}} k+{k_{1}} p
   k+{k_{1}}^2-{k_{1}}\big) {x_{23}}^8 
   \\-2 {k_{1}}^2 (p k-k+{k_{1}}-2) \big(2 p^2 k^2-6 p
   k^2+4 k^2-2 {k_{1}} k+2 {k_{1}} p
   k+{k_{1}}^2-{k_{1}}\big) {x_{23}}^7\\
   +{k_{1}}
   \big(2 p^4 k^4-10 p^3 k^4+18 p^2 k^4-14 p k^4+4 k^4+10
   {k_{1}} p^3 k^3-4 p^3 k^3-28 {k_{1}} p^2 k^3\\+14 p^2
   k^3-2 {k_{1}} k^3+20 {k_{1}} p k^3-14 p k^3+4 k^3+5
   {k_{1}}^2 k^2+15 {k_{1}}^2 p^2 k^2-35 {k_{1}} p^2
   k^2\\
   +4 p^2 k^2-41 {k_{1}} k^2-23 {k_{1}}^2 p k^2+79
   {k_{1}} p k^2-12 p k^2+8 k^2-2 {k_{1}}^3 k+39
   {k_{1}}^2 k\\-33 {k_{1}} k+8 {k_{1}}^3 p k-40
   {k_{1}}^2 p k+28 {k_{1}} p k+{k_{1}}^4-16
   {k_{1}}^3+19 {k_{1}}^2-4 {k_{1}}\big)
   {x_{23}}^6
              \end{array} }
\end{equation*}
 \begin{equation*}
{\small \begin{array}{l}
   -2 {k_{1}} (p k-k+{k_{1}}-2) \big(4
   p^3 k^3-12 p^2 k^3+10 p k^3-2 k^3+6 {k_{1}} p^2 k^2-8
   p^2 k^2
   -2 {k_{1}} k^2\\-4 {k_{1}} p k^2+18 p k^2-10
   k^2+{k_{1}}^2 k+15 {k_{1}} k+3 {k_{1}}^2 p k-19
   {k_{1}} p k +4 p k-4 k-8 {k_{1}}^2+8
   {k_{1}}\big) {x_{23}}^5\\
   +\big(p^5 k^5-5 p^4
   k^5+10 p^3 k^5-10 p^2 k^5+5 p k^5-k^5+14 {k_{1}} p^4
   k^4-4 p^4 k^4-38 {k_{1}} p^3 k^4\\
   +14 p^3 k^4+24
   {k_{1}} p^2 k^4-18 p^2 k^4-10 {k_{1}} k^4+10
   {k_{1}} p k^4+10 p k^4-2 k^4+25 {k_{1}}^2 p^3
   k^3\\
   -83 {k_{1}} p^3 k^3+8 p^3 k^3+8 {k_{1}}^2 k^3-36
   {k_{1}}^2 p^2 k^3+204 {k_{1}} p^2 k^3-20 p^2 k^3+37
   {k_{1}} k^3\\+3 {k_{1}}^2 p k^3-158 {k_{1}} p
   k^3+17 p k^3-5 k^3-3 {k_{1}}^3 k^2-38 {k_{1}}^2
   k^2+14 {k_{1}}^3 p^2 k^2-122 {k_{1}}^2 p^2 k^2\\
   +152
   {k_{1}} p^2 k^2-8 p^2 k^2+101 {k_{1}} k^2-2
   {k_{1}}^3 p k^2+150 {k_{1}}^2 p k^2-252 {k_{1}} p
   k^2+12 p k^2\\
   -4 k^2+12 {k_{1}}^3 k-116 {k_{1}}^2
   k+100 {k_{1}} k+2 {k_{1}}^4 p k-60 {k_{1}}^3 p
   k+174 {k_{1}}^2 p k\\-112 {k_{1}} p k+4 p k-4 k-4
   {k_{1}}^4+60 {k_{1}}^3-76 {k_{1}}^2+20
   {k_{1}}\big) {x_{23}}^4\\
   -2 (p k-k+{k_{1}}-2)
   \big(2 p^4 k^4-6 p^3 k^4+6 p^2 k^4-2 p k^4+6 {k_{1}}
   p^3 k^3-8 p^3 k^3-2 {k_{1}} p^2 k^3\\
   +18 p^2 k^3+4
   {k_{1}} k^3-8 {k_{1}} p k^3-14 p k^3+4 k^3-2
   {k_{1}}^2 k^2+3 {k_{1}}^2 p^2 k^2 
   -35 {k_{1}} p^2 k^2\\+12 p^2 k^2 -2 {k_{1}} k^2+2 {k_{1}}^2 p k^2+34
   {k_{1}} p k^2-16 p k^2+4 k^2+2 {k_{1}}^2 k-34
   {k_{1}} k\\-16 {k_{1}}^2 p k+48 {k_{1}} p k-8 p k+8
   k+16 {k_{1}}^2-16 {k_{1}}\big)
   {x_{23}}^3\\
   +\big(6 p^5 k^5-16 p^4 k^5+8 p^3 k^5+12
   p^2 k^5-14 p k^5+4 k^5+13 {k_{1}} p^4 k^4-49 p^4
   k^4\\
   -15 {k_{1}} p^3 k^4+127 p^3 k^4-7 {k_{1}} p^2
   k^4-109 p^2 k^4+2 {k_{1}} k^4+7 {k_{1}} p k^4
   +33 p k^4-2 k^4\\+8 {k_{1}}^2 p^3 k^3-96 {k_{1}} p^3
   k^3+140 p^3 k^3-4 {k_{1}}^2 k^3+2 {k_{1}}^2 p^2
   k^3+117 {k_{1}} p^2 k^3-283 p^2 k^3\\
   -26 {k_{1}}
   k^3-6 {k_{1}}^2 p k^3+5 {k_{1}} p k^3+169 p k^3-26
   k^3+32 {k_{1}}^2 k^2+{k_{1}}^3 p^2 k^2\\
   -56
   {k_{1}}^2 p^2 k^2+243 {k_{1}} p^2 k^2-188 p^2
   k^2+44 {k_{1}} k^2+8 {k_{1}}^2 p k^2-268 {k_{1}}
   p k^2\\
   +276 p k^2-92 k^2-56 {k_{1}}^2 k+164 {k_{1}}
   k-4 {k_{1}}^3 p k+120 {k_{1}}^2 p k-240 {k_{1}} p
   k\\+108 p k-92 k+4 {k_{1}}^3-64 {k_{1}}^2+76
   {k_{1}}-16\big) {x_{23}}^2\\
   -2 (k p-2) (p
   k-k+{k_{1}}-2) (2 p k-2 k+{k_{1}}-1) \big(p^2
   k^2+p k^2-2 k^2-6 p k+4 k+4\big) {x_{23}}\\
   +(p
   k-k+{k_{1}}-1) \big(p^2 k^2+p k^2-2 k^2-6 p k+4
   k+4\big)^2=0. 
\end{array} }
\end{equation*}
Now we observe that 
 \begin{equation}\label{sakane:values_h(x23)}
{\small \begin{array}{l}
H_1(0) 
   = 2 \big(k^2 (p-1) (p+2)+k (3 p-2)+1\big)^2 (2 k (p-1)+2 {k_1}+1),\\ \\
H_1(1)=(k-1) (k- {k_1}) (k p-k+ {k_1}-1) (k p-k+{k_1})^2),\\ \\ 
H_1(2) =64 ({k_{1}}-k)^5+ ({k_1}-k)^4 ((k-3) (96 (p-3)+416)+288
   (p-3)+1184)\\
   +({k_{1}}-k)^3 \big((k-3)^2 \big(52
   (p-3)^2+568 (p-3)+1092\big)+(k-3)  (312
   (p-3)^2\\
   +3264 (p-3)+6408)+468 (p-3)^2+4680
   (p-3)+9284\big)\\
   +({k_{1}}-k)^2 \big((k-3)^3
   \big(28 (p-3)^3+268 (p-3)^2+1268
   (p-3)+1476\big)\\
   +(k-3)^2 \big(252 (p-3)^3+2460
   (p-3)^2+11068 (p-3)+13340\big)\\
   +(k-3) \big(756
   (p-3)^3+7524 (p-3)^2
   +32428 (p-3)+40012\big)\\
   +756
   (p-3)^3+7668 (p-3)^2+31908
   (p-3)+39924\big)\\
   +({k_{1}}-k) \big((k-3)^4 \big(9
   (p-3)^4+106 (p-3)^3+477 (p-3)^2+1292
   (p-3)+1040\big)\\
   +(k-3)^3 \big(108 (p-3)^4+1304
   (p-3)^3+5968 (p-3)^2+15468 (p-3)+12872\big)\\
   +(k-3)^2
   \big(486 (p-3)^4+6012 (p-3)^3+28010 (p-3)^2+70164
   (p-3)+59840\big)\\+(k-3) \big(972 (p-3)^4+12312
   (p-3)^3+58440 (p-3)^2+142804 (p-3)+123880\big)\\
   +729
   (p-3)^4+9450 (p-3)^3+45729 (p-3)^2+109920
   (p-3)+96416\big)\\
   +(k-3)^5
   \big((p-3)^5+17 (p-3)^4+103 (p-3)^3+299 (p-3)^2+524
   (p-3)+312\big)\\
   +(k-3)^4 \big(15 (p-3)^5+258
   (p-3)^4+1599 (p-3)^3+4740 (p-3)^2+8092
   (p-3)+4964\big)\\
   +(k-3)^3 \big(90 (p-3)^5+1566
   (p-3)^4+9918 (p-3)^3+30030 (p-3)^2\\
   +50420
   (p-3)+31712\big)+(k-3)^2 \big(270 (p-3)^5+4752
   (p-3)^4+30726 (p-3)^3\\
   +95028 (p-3)^2+158164
   (p-3)+101508\big)+(k-3) \big(405 (p-3)^5+7209
   (p-3)^4\\
   +47547 (p-3)^3+150183 (p-3)^2+249376
   (p-3)+162656\big)+243 (p-3)^5\\
   +4374 (p-3)^4+29403
   (p-3)^3+94824 (p-3)^2+157872 (p-3)+104336. 
\end{array} }
\end{equation}
From the above, for $k_1 \geq k \geq 3$ and  $ p \geq 3$, we see that
$$H_1(0) > 0,  \quad  H_1(2) > 0, $$ 
and,  for $k_1 > k $,  $H_1(1) < 0$. 


Hence,  it follows that the equation $H_1(x_{23})=0$ has at least two positive solutions $x_{23}=\alpha$ and $x_{23}=\beta$  with 
$0 < \alpha < 1< \beta < 2$. Also, from (\ref{eq_x1_rat}) we  see that the values of $x_1$ corresponding to these values $x_{23}=\alpha$ and $x_{23}=\beta$  are real numbers.

For $k_1=k$ we have that $H_1(x_{23})= (x_{23}-1) G_1(x_{23})$, where 
 \begin{equation*}\label{eqx23_k1=k}
\begin{array}{l}
G_1(x_{23})=k^4 (p-1) {x_{23}}^7 \big(k p^2-2 k p+2 k-1\big)\\
-k^3 {x_{23}}^6 \big(3 k^2 p^3-5
   k^2 p^2+2 k^2 p+2 k^2-8 k p^2+15 k p-13 k+4\big)\\
   +k^2 {x_{23}}^5 \big(2 k^3
   p^4-3 k^3 p^3+10 k^3 p^2-11 k^3 p+4 k^3-4 k^2 p^3\\
   -13 k^2 p^2+10 k^2
   p-k^2+4 k p^2+16 k p-10 k-4\big)\\
   -k^2 {x_{23}}^4 \big(6 k^3 p^4-9 k^3
   p^3+8 k^3 p^2+5 k^3 p-4 k^3-28 k^2 p^3+35 k^2 p^2\\
   -52 k^2 p+13 k^2+36 k
   p^2-4 k p+22 k-16 p-12\big)\\
   +k {x_{23}}^3 \big(k^4 p^5+3 k^4 p^4+6 k^4
   p^3-16 k^4 p^2+13 k^4 p-2 k^4-4 k^3 p^4\\
   -41 k^3 p^3+29 k^3 p^2-6 k^3 p-8
   k^3+8 k^2 p^3+96 k^2 p^2-57 k^2 p\\
   +18 k^2-8 k p^2-84 k p+32 k+4
   p+16\big)\\
   -k {x_{23}}^2 \big(3 k^4 p^5-3 k^4 p^4+8 k^4 p^3-9 k^4 p+2
   k^4-20 k^3 p^4+7 k^3 p^3-49 k^3 p^2\\
   +46 k^3 p+48 k^2 p^3+36 k^2 p^2+13 k^2
   p-34 k^2-56 k p^2-60 k p+24 k+28 p+16\big)\\
   +{x_{23}} \big(k^2
   p^2+k^2 p-2 k^2-6 k p+4 k+4\big) \big(3 k^3 p^3-3 k^3 p^2+2 k^3 p-11
   k^2 p^2\\
   +5 k^2 p-2 k^2+14 k p-4 k-4\big)+(1-k p) \big(k^2
   p^2+k^2 p-2 k^2-6 k p+4 k+4\big)^2.
\end{array} 
\end{equation*}
Now we observe that, for $p \geq3$ and $k\geq3$,  
$$G_1(1)= -2 (k-2) (k-1) k p \big(k^2 p^2-2 k p-k+2\big) < 0, \quad G_1(2) = H_1(2) >0.$$
Therefore, equation $G_1(x_{23})=0$ has at least one positive solution $x_{23}=\gamma$   with 
$1 < \gamma  < 2$. From (\ref{eq_x1_rat}), we also see that the value of $x_1$ corresponding to the value $x_{23}=\gamma $ is  a real number. 

%

Now we show that, for positive solutions $x_2$, $x_{23}$ above, the value  of solution  $x_1$ is positive. 
Note that the value $x_1$ given by (\ref{eq_x1_rat}) also satisfies  the equation $g_1=0$ in (\ref{eq26g}) and that 
 $g_1$  is a quadratic polynomial with respect to $x_1$.  We consider the two solutions of the quadratic equation $g_1=0$ with respect to $x_1$. 
We can write the  quadratic polynomial $g_1$  with respect to $x_1$ as 
\begin{eqnarray}\label{quadratic_g1}
 &&
g_1= k (p-1) {x_2} {x_{23}}^2
   \left({x_1}-\frac{k (p-2) {x_2}^2+(k-2) {x_{23}}^2+{k_1} {x_2}^2 {x_{23}}^2}{2 k (p-1) {x_2} {x_{23}}^2}\right)^2 \nonumber
   \\&& - \frac{\left(k (p-2) {x_2}^2+(k-2) {x_{23}}^2+{k_1} {x_2}^2
   {x_{23}}^2\right)^2}{4 k (p-1) {x_2} {x_{23}}^2}+({k_1}-2) {x_2} {x_{23}}^2. 
\end{eqnarray}
We see  
 the value $\displaystyle \frac{k (p-2) {x_2}^2+(k-2) {x_{23}}^2+{k_1} {x_2}^2 {x_{23}}^2}{2 k (p-1) {x_2} {x_{23}}^2}$  is positive, whenever $x_2$, $x_{23}$ are positive values,  and the value of $g_1$ at $x_1=0$ is given by $({k_1}-2) {x_2} {x_{23}}^2 > 0$. Therefore, if  the quadratic equation $g_1=0$ has real solutions, these are positive.  Hence, we see that real solutions $x_1$ are positive. 
   \end{proof}

\begin{remark} 
  For $k_1 = 3, k=4, p=4$, we see that 
\begin{equation*}
\begin{array}{l} 
H_1(x_{23})= 13662 {x_{23}}^8-63180 {x_{23}}^7+247464
   {x_{23}}^6-698256 {x_{23}}^5+1424064
   {x_{23}}^4\\-2312544 {x_{23}}^3+2775392
   {x_{23}}^2-2006368 {x_{23}}+629216
      \end{array}
   \end{equation*} and the equation $H_1(x_{23})=0$ has no real solutions. Thus, in general, we need the condition $k_1 \geq k$ to show the existence of non naturally reductive Einstein metric on $\SO(n)$  of the form  {\em(\ref{metric009})}. 
\end{remark}

\begin{theorem}\label{thm6.3}
 Let $p\ge 3$ and $k\geq 3$.  If  ${k_{1}} \geq 10 k p$, then  the Lie group  $\SO(n)$ {\em ($n=k_1 +(p-1)k$) }admits at least four  $\Ad( \SO(k_1)\times(\SO(k))^{p-1})$-invariant Einstein metrics of the form {\em (\ref{metric009})}, which are not naturally reductive. 
\end{theorem}
    \begin{proof} 
  In the proof of Theorem \ref{thm6.2} we showed that the equation $H_1(x_{23}) =0$   has at least two solutions.  So we will first prove that $H_1(x_{23}) =0$ has two more solutions. 
We claim that $H_1(2/3) >0$ for ${k_{1}} \geq 10 k p$, $k \geq 3$ and $p \geq 3$. 
Indeed, we see that 
\begin{equation}\label{H_1(2/3)}
H_1(2/3) = \frac{1}{3^8}\Big(\sum_{j=0}^{4} p_j(k,   p)({k_{1}}-10 k p)^j+ 64 ({k_{1}}-10 k p)^5\Big),
\end{equation}
where 
\begin{equation*}\label{p_j(k, p)}
{ \begin{array}{l}
p_0(k,   p) = \big(8477931 (k-3)^5+127168965 (k-3)^4+763013790 (k-3)^3\\
+2289041370
   (k-3)^2+3433562055 (k-3)+2060137233\big) (p-3)^5\\
   +\big(105475930
   (k-3)^5+1609348993 (k-3)^4+9819354216 (k-3)^3\\
   +29947843422 (k-3)^2+45656436294
   (k-3)+27834664473\big) (p-3)^4\\
   +\big(515065215 (k-3)^5+8036819231
   (k-3)^4 +50080603950 (k-3)^3\\
   +155804805126 (k-3)^2+242027588979
   (k-3)+150194316987\big) (p-3)^3\\
   +\big(1225798700 (k-3)^5+19719309159
   (k-3)^4+126263852928 (k-3)^3\\
   +402496759170 (k-3)^2+639093441636
   (k-3)+404538563655\big) (p-3)^2\\
   +\big(1406250300 (k-3)^5+23633211600
   (k-3)^4+156904304004 (k-3)^3\\
   +515627980956 (k-3)^2+840207388320
   (k-3)+543815227764\big) (p-3)\\
   +609894432 (k-3)^5+10964539404
   (k-3)^4+76557997800 (k-3)^3\\
   +261601463052 (k-3)^2+439706500272
   (k-3)+291848917536, \\
   p_1(k,   p) = \big(4056465 (k-3)^4+48677580 (k-3)^3+219049110 (k-3)^2\\
   +438098220
   (k-3)+328573665\big) (p-3)^4
   +\big(40943714 (k-3)^4\\
   +500737232
   (k-3)^3
   +2295674532 (k-3)^2+4676063040 (k-3)\\+3570582762\big)
   (p-3)^3
   +\big(152710729 (k-3)^4+1913265216 (k-3)^3\\+8971418610
   (k-3)^2 
   +18663109560 (k-3)+14535152973\big) (p-3)^2\\+\big(248336188
   (k-3)^4
   +3210990756 (k-3)^3+15479652636 (k-3)^2\\+33001360668
   (k-3)
   +26268727464\big) (p-3) 
   +147552132 (k-3)^4\\+1990966824 (k-3)^3
   +9937811016
   (k-3)^2+21806079336 (k-3)\\+17782338564, \\
     p_2(k,   p) = \big(771308 (k-3)^3+6941772 (k-3)^2+20825316 (k-3)\\
  +20825316\big) (p-3)^3+\big(5911096 (k-3)^3+54414120 (k-3)^2\\
  +166885128 (k-3) +170527896\big) (p-3)^2+\big(14927996 (k-3)^3\\
  +141302436
   (k-3)^2+444605220 (k-3)+465149628\big) (p-3)\\
   +12385712 (k-3)^3+121422408
   (k-3)^2+393679584 (k-3)+422648136,\\
         \end{array}}
  \end{equation*}
\begin{equation*}
{ \begin{array}{l}
  p_3(k,   p) = \big(72884 (k-3)^2+437304 (k-3)+655956\big) (p-3)^2\\
  +\big(376440
   (k-3)^2+2327872 (k-3)+3595656\big) (p-3)+481732 (k-3)^2\\
   +3088712  (k-3)+4925700,
  \\
  p_4(k,   p) = (3424 (k-3)+10272) (p-3)+8928 (k-3)+28256. 
      \end{array}
      }
   \end{equation*} 
   Thus we see $p_j(k,   p) > 0$ for $j =0, 1, 2, 3, 4$ and  $H_1(2/3) >0$ for ${k_{1}} \geq 10 k p$, $k \geq 3$ and $p \geq 3$. 
   
   \smallskip
   
   Now we claim that $H_1\Big(\displaystyle\frac{k^2 p^2+k^2 p-2 k^2-6 k p+4 k+4}{{k_1} (k p-2)}\Big) < 0$ for ${k_{1}} \geq  k p$, $k \geq 3$ and $p \geq 3$. Note that $k^2 p^2+k^2 p-2 k^2-6 k p+4 k+4 = (k (p-1)-2) (k (p+2)-4)+4 (k-1) > 0$. 
  
By the same argument as  in $H_1(2/3)$,  we see that 
   \begin{equation*}
   \begin{array}{l}H_1\Big(\displaystyle\frac{k^2 p^2+k^2 p-2 k^2-6 k p+4 k+4}{{k_1} (k p-2)}\Big) \\
\displaystyle   = -\frac{\big(k^2 p^2+k^2 p-2 k^2-6 k p+4 k+4\big)^2}{{k_1}^6 (k p-2)^8}\Big(\sum_{j=0}^{6} q_j(k,   p)({k_{1}}-  k p)^j \Big),
     \end{array}
   \end{equation*}
where 
 $q_j(k,   p) > 0$ ($j =0, 1,\ldots,  6$) for  $k \geq 3$ and $p \geq 3$.  
 
Note that $\displaystyle\frac{k^2 p^2+k^2 p-2 k^2-6 k p+4 k+4}{{k_1} (k p-2)} < \frac{2}{3}$ for ${k_{1}} \geq  10 k p$, $k \geq 3$ and $p \geq 3$. 

In fact, \begin{equation*}
\begin{array}{l}
\displaystyle\frac{2}{3} - \frac{k^2 p^2+k^2 p-2 k^2-6 k p+4 k+4}{{k_1} (k p-2)} \\
\displaystyle= \frac{17 k^2 (p-3)^2+\left(99 k^2-22 k\right) (p-3)+(2 k p-4)
   ({k_1}-10 k p)+150 k^2-78 k-12}{3 {k_1} (k p-2)}>0
      \end{array} 
\end{equation*}
Thus, we see that there are four  solutions $\alpha_i$ $(i =1,2,3,4)$ of the equation $H_1(x_{23})  = 0$ with
   $$0 < \alpha_1 <  \displaystyle\frac{k^2 p^2+k^2 p-2 k^2-6 k p+4 k+4}{{k_1} (k p-2)} < \alpha_2 < \frac{2}{3} < \alpha_3  < 1 < \alpha_4 < 2 $$
 for ${k_{1}} \geq 10 k p$, $k \geq 3$ and $p \geq 3$. 
    \end{proof}

\end{document}